\documentclass[10pt]{article}

\usepackage{amsmath}
\usepackage{amssymb}
\usepackage{amsthm}
\usepackage{graphics}
\usepackage[latin1]{inputenc}
\usepackage[T1]{fontenc}

\usepackage{enumerate}

\newtheorem{thm}{Theorem}[section]
\newtheorem{prop}{Proposition}[section]

\newtheorem{lemma}{Lemma}[section]
\newtheorem{rem}{Remark}[section]

\newtheorem{defi}{Definition}[section]

\newcommand{\R}{\mathbb{R}}             
\newcommand{\N}{\mathbb{N}}             
\newcommand{\C}{\mathbb{C}}             

\newcommand{\Section}[1]{\section{#1} \setcounter{equation}{0}}

\begin{document}

\title{A survey of non-uniqueness results for the anisotropic Calder\'on problem with disjoint data}
\author{Thierry Daud\'e \footnote{Research supported by the French National Research Projects AARG, No. ANR-12-BS01-012-01, and Iproblems, No. ANR-13-JS01-0006} $^{\,1}$, Niky Kamran \footnote{Research supported by NSERC grant RGPIN 105490-2011} $^{\,2}$ and Francois Nicoleau \footnote{Research supported by the French National Research Project NOSEVOL, No. ANR- 2011 BS0101901} $^{\,3}$\\[12pt]
 $^1$  \small D\'epartement de Math\'ematiques. UMR CNRS 8088, Universit\'e de Cergy-Pontoise, \\
 \small 95302 Cergy-Pontoise, France. \\
\small Email: thierry.daude@u-cergy.fr \\
$^2$ \small Department of Mathematics and Statistics, McGill University,\\ \small  Montreal, QC, H3A 2K6, Canada. \\
\small Email: nkamran@math.mcgill.ca \\
$^3$  \small  Laboratoire de Math\'ematiques Jean Leray, UMR CNRS 6629, \\ \small 2 Rue de la Houssini\`ere BP 92208, F-44322 Nantes Cedex 03. \\
\small Email: francois.nicoleau@math.univ-nantes.fr }





\maketitle


\begin{abstract}

After giving a general introduction to the main known results on the anisotropic Calder\'on problem on $n$-dimensional compact Riemannian manifolds with boundary, we give a motivated review of some recent non-uniqueness results obtained in \cite{DKN2, DKN3} for the anisotropic Calder\'on problem at fixed frequency, in dimension $n\geq 3$, when the Dirichlet and Neumann data are measured on disjoint subsets of the boundary. These non-uniqueness results are of the following nature: given a smooth compact connected Riemannian manifold with boundary $(M,g)$ of dimension $n\geq 3$, we first show that there exist in the conformal class of $g$ an infinite number of Riemannian metrics $\tilde{g}$ such that their corresponding Dirichlet-to-Neumann maps at a fixed frequency coincide when the Dirichlet data $\Gamma_D$ and Neumann data $\Gamma_N$ are measured on disjoint sets and satisfy $\overline{\Gamma_D \cup \Gamma_N} \ne \partial M$. The corresponding conformal factors satisfy a nonlinear elliptic PDE of Yamabe type on $(M,g)$ and arise from a natural but subtle gauge invariance of the Calder\'on when the data are given on disjoint sets. We then present counterexamples to uniqueness in dimension $n\geq 3$ to the anisotropic Calder\'on problem at fixed frequency with data on disjoint sets, which do not arise from this gauge invariance. They are given by cylindrical Riemannian manifolds with boundary having two ends, equipped with a suitably chosen warped product metric. This survey concludes with some remarks on the case of manifolds with corners.



\vspace{0.5cm}




\end{abstract}

\tableofcontents


\Section{Introduction} \label{0}


The anisotropic Calder\'on problem is a problem of geometric analysis that originates in the important physical question of determining whether one can recover properties such as the electrical conductivity of a medium by making measurements at its boundary. The Calder\'on problem is still far from being completely understood, especially where issues of non-uniqueness are concerned \cite{GT2, IUY2, KS1, KS2, KLO, LO1, LO2}. Our goal in this paper is to give a motivated account of some non-uniqueness results that have been recently obtained in \cite{DKN2,DKN3} for the Calder\'on problem in the case in which the Dirichlet and Neumann data are measured on disjoint subsets of the boundary. At the same time, we will also give a survey of the main uniqueness results that have been obtained so far on Calder\'on problem in the general setting of Riemannian manifolds with boundary. As a complement to the review provided in this paper, we refer to the surveys \cite{GT2, KS2, Sa, U1} for a description of the current state of the art on the general anisotropic Calder\'on problem and also to \cite{DSFKSU, DSFKLS, GSB, GT1, KS1, LaTU, LaU, LeU} for important contributions to the question of uniqueness.

The anisotropic Calder\'on problem can be naturally formulated as a problem of geometric analysis in terms of the Dirichlet-to-Neumann map, or for short the DN map, for the Laplacian on Riemannian manifolds with boundary. We first recall the definition of the DN map in the general formulation that has been given by Lee and Uhlmann \cite{LeU}. Let $(M, g)$ denote an $n$-dimensional smooth compact connected Riemannian manifold with smooth boundary $\partial M$, and let $\Delta_{LB}$ be the positive Laplace-Beltrami operator on $(M,g)$, given in local coordinates by
$$
\Delta_{LB}=  -\Delta_g = -\frac{1}{\sqrt{|g|}} \partial_i \left( \sqrt{|g|} g^{ij} \partial_j \right).
$$
It is standard (see for instance \cite{KKL}) that the Laplace-Beltrami operator $-\Delta_g$ with Dirichlet boundary conditions on $\partial M$ is self-adjoint on $L^2(M, dVol_g)$ and has pure point spectrum $\{ \lambda_j\}_{j \geq 1}$ with $0 < \lambda_1 < \lambda_2 \leq \dots \leq \lambda_j \to +\infty$.

We consider the Dirichlet problem
\begin{equation} \label{Eq00}
  \left\{ \begin{array}{cc} -\Delta_g u = \lambda u, & \textrm{on} \ M, \\ u = \psi, & \textrm{on} \ \partial M. \end{array} \right.
\end{equation}
where the frequency $\lambda \in \R$ is assumed to lie outside the Dirichlet spectrum, that is $\lambda \notin \{ \lambda_j\}_{j \geq 1}$. We know (see for instance \cite{Sa, Ta1}) that for any such $\lambda$  and that for any $\psi \in H^{1/2}(\partial M)$, there exists a unique weak solution $u \in H^1(M)$ of the Dirichlet problem (\ref{Eq00}). This allows us to define the Dirichlet-to-Neumann (DN) map as the operator $\Lambda_{g}(\lambda)$ from $H^{1/2}(\partial M)$ to $H^{-1/2}(\partial M)$ given by
\begin{equation} \label{DN-Abstract}
  \Lambda_{g}(\lambda) (\psi) = \left( \partial_\nu u \right)_{|\partial M},
\end{equation}
where $u$ is the unique solution of (\ref{Eq00}) and $\left( \partial_\nu u \right)_{|\partial M}$ is its normal derivative with respect to the unit outer normal $\nu$ on $\partial M$. The latter is defined in the weak sense as an element of $H^{-1/2}(\partial M)$ by
$$
  \left\langle \Lambda_{g}(\lambda) \psi | \phi \right \rangle = \int_M \langle du, dv \rangle_g \, dVol_g,
$$
where $\psi \in H^{1/2}(\partial M)$ and $\phi \in H^{1/2}(\partial M)$, where $u$ is the unique solution of the Dirichlet problem (\ref{Eq00}), and where $v$ is any element of $H^1(M)$ such that $v_{|\partial M} = \phi$. When $\psi$ is sufficiently smooth, this definition coincides with the usual one in local coordinates, that is
\begin{equation} \label{DN-Coord}
\partial_\nu u = \nu^i \partial_i u.
\end{equation}


As mentioned earlier, we are interested in the case in which the Dirichlet and Neumann data are measured on disjoint subsets of the boundary, and we are therefore led to introduce the \emph{partial} DN maps, which are defined as follows. Let $\Gamma_D$ and $\Gamma_N$ denote open subsets of $\partial M$. The partial DN map $\Lambda_{g,\Gamma_D,\Gamma_N}(\lambda)$ is defined as the DN map $\Lambda_g(\lambda)$ restricted to the case in which the Dirichlet data are prescribed on $\Gamma_D$ and the Neumann data are measured on $\Gamma_N$. More precisely, consider the Dirichlet problem
\begin{equation} \label{Eq0}
  \left\{ \begin{array}{cc} -\Delta_g u = \lambda u, & \textrm{on} \ M, \\ u = \psi, & \textrm{on} \ \Gamma_D, \\ u = 0, & \textrm{on} \ \partial M \setminus \Gamma_D. \end{array} \right.
\end{equation}
We define $\Lambda_{g,\Gamma_D,\Gamma_N}(\lambda)$ as the operator acting on functions $\psi \in H^{1/2}(\partial M)$ with $\textrm{supp}\,\psi \subset \Gamma_D$ by
\begin{equation} \label{Partial-DNmap}
  \Lambda_{g,\Gamma_D,\Gamma_N}(\lambda) (\psi) = \left( \partial_\nu u \right)_{|\Gamma_N},
\end{equation}
where $u$ is the unique solution of (\ref{Eq0}).

The anisotropic partial Calder\'on problem can now be stated as follows in its raw form: \emph{If a pair of partial DN maps $\Lambda_{g_1,\Gamma_D, \Gamma_N}(\lambda)$ and $\Lambda_{g_2,\Gamma_D, \Gamma_N}(\lambda)$ coincide at a fixed frequency $\lambda$, can one conclude that the metrics $g_1$ and $g_2$ are the same}?

There are a number of natural gauge invariances for this problem which are of geometric origin and which imply that the answer to the question stated above is necessarily going to be negative. These lead to refined formulations of the Calder\'on problem that we shall present shortly, and that constitute the actual statement of this inverse problem. Before doing so, let us review the gauge invariances in question. First, it results from the definition (\ref{Eq0}) - (\ref{Partial-DNmap}) that the partial DN map $\Lambda_{g, \Gamma_D, \Gamma_N}(\lambda)$ is invariant when the metric $g$ is pulled back by any diffeomorphism of $M$ that restrict to the identity on $\Gamma_D \cup \Gamma_N$, \textit{i.e.}
\begin{equation} \label{Inv-Diff}
  \forall \phi \in \textrm{Diff}(M) \ \textrm{such that} \ \phi_{|\Gamma_D \cup \Gamma_N} = Id, \quad \Lambda_{\phi^*g, \Gamma_D, \Gamma_N}(\lambda) = \Lambda_{g, \Gamma_D, \Gamma_N}(\lambda).
\end{equation}

In dimension two and for zero frequency $\lambda = 0$, the scaling action induced on the Laplacian by conformal changes of metric leads to an additional gauge invariance of the DN map that applies to this specific setting. Indeed, recall that if $\dim M=2$, then
$$
\Delta_{cg} = \frac{1}{c} \Delta_g,
$$
for any smooth function $c >0$, so that
\begin{equation} \label{Inv-Conf}
\forall c \in C^\infty(M) \ \textrm{such that} \ c >0 \ \textrm{and} \ c_{|\Gamma_N} = 1, \quad \Lambda_{c g, \Gamma_D, \Gamma_N}(0) = \Lambda_{g, \Gamma_D, \Gamma_N}(0),
\end{equation}
since the unit outer normal vectors $\nu_{cg}$ and $\nu_g$ are identical on $\Gamma_N$.

From the preceding remarks, it follows that the gauge-invariant formulation of our inverse problem, which is referred to as the \emph{anisotropic Calder\'on conjecture}, is the following. \\

\noindent \textbf{(Q1)}: \emph{Let $M$ be a smooth compact connected manifold with smooth boundary $\partial M$ and let $g,\, \tilde{g}$ denote smooth Riemannian metrics on $M$ and let $\Gamma_D, \Gamma_N$ be open subsets of $\partial M$. Assume that $\lambda \in \R$ does not belong to $\sigma(-\Delta_g) \cup \sigma(-\Delta_{\tilde{g}})$ and suppose that
$$
  \Lambda_{g,\Gamma_D, \Gamma_N}(\lambda) = \Lambda_{\tilde{g},\Gamma_D, \Gamma_N}(\lambda).
$$
Does it follow that
$$
  g = \tilde{g},
$$
up to the gauge invariance (\ref{Inv-Diff}) if $\dim M \geq 3$ and up to the gauge invariances (\ref{Inv-Diff}) - (\ref{Inv-Conf}) if $\dim M = 2$ and $\lambda = 0$}? \\

Several subcases of the above problem may naturally be considered:
\begin{itemize}
\item \textbf{Full data}: $\Gamma_D = \Gamma_N = \partial M$, in which case, we denote the DN map simply by $\Lambda_g(\lambda)$.

\item \textbf{Local data}: $\Gamma_D = \Gamma_N = \Gamma$, where $\Gamma$ can be any nonempty open subset of $\partial M$. In that case, we denote the DN map by $\Lambda_{g, \Gamma}(\lambda)$.

\item \textbf{Data on disjoint sets}: $\Gamma_D$ and $\Gamma_N$ are disjoint open sets of $\partial M$.
\end{itemize}

If $\dim M \geq 3$, one may also consider an inverse problem of a different and simpler nature by assuming that the Riemannian manifolds $(M,g)$ and $(M,\tilde{g})$ belong to the same conformal class, that is $\tilde{g} = c g$ for some strictly positive smooth function c. We thus think of $g$ as a given background metric and the problem is to recover the unknown conformal factor $c$ from the DN map $\Lambda_{c g,\Gamma_D, \Gamma_N}(\lambda)$. In that case, the statement of anisotropic Calder\'on problem reduces to the following: \\

\noindent \textbf{(Q2)}: \emph{Let $(M,g)$ be a smooth compact connected Riemannian manifold of dimension $n\geq 3$ with smooth boundary $\partial M$ and let $\Gamma_D, \Gamma_N$ denote open subsets of $\partial M$. Let $c$ be a smooth strictly positive function on $M$ and assume that $\lambda \in \R$ does not belong to $\sigma(-\Delta_g) \cup \sigma(-\Delta_{c g})$. If
$$
  \Lambda_{c g,\Gamma_D, \Gamma_N}(\lambda) = \Lambda_{g,\Gamma_D, \Gamma_N}(\lambda),
$$
does there exist a diffeomorphism $\phi: \, M \longrightarrow M$ with $\phi_{| \, \Gamma_D \cup \Gamma_N} = Id$ such that}
\begin{equation} \label{Inv-Conformal}
  \phi^* g = c g?
\end{equation}
It is important to note that since any diffeomorphism $\phi: \, M \longrightarrow M$  which satisfies $\phi^* g = c g$ and $\phi_{|\Gamma} = Id$ for a non-empty open subset $\Gamma$ of $\partial M$ must be the identity \cite{Li},
there is no ambiguity arising from the diffeomorphism invariance of the DN map in the solution of the anisotropic Calder\'on problem \textbf{(Q2)}. The condition (\ref{Inv-Conformal}) may therefore be replaced by the condition
\begin{equation} \label{Inv-Conformal-1}
  c = 1, \quad \textrm{on} \ M.
\end{equation}

One may also extend the scope of the anisotropic Calder\'on problem to include the presence of an external potential. We shall see in Proposition \ref{Link-c-to-V} below that this question bears a close relation to \textbf{(Q2)}. We thus consider the time-independent Schr\"odinger equation on $(M,g)$ with a potential $V \in L^\infty(M)$
\begin{equation} \label{Eq0-Schrodinger}
  \left\{ \begin{array}{cc} (-\Delta_g + V) u = \lambda u, & \textrm{on} \ M, \\ u = \psi, & \textrm{on} \ \Gamma_D, \\ u = 0, & \textrm{on} \ \partial M \setminus \Gamma_D. \end{array} \right.
\end{equation}
If $\lambda$ does not belong to the Dirichlet spectrum of $-\Delta_g +V$, then it is a standard result that for any $\psi \in H^{1/2}(\partial M)$, there exists a unique weak solution $u \in H^1(M)$ of (\ref{Eq0-Schrodinger}) (see for example \cite{DSFKSU, Sa}). We thus have a partial Dirichlet-to-Neumann map $\Lambda_{g, V, \,\Gamma_D, \Gamma_N}(\lambda)$ for all $\psi \in H^{1/2}(\partial M)$ with supp $\psi \subset \Gamma_D$, defined by
\begin{equation} \label{DN-Abstract-Schrodinger}
  \Lambda_{g, V,\Gamma_D, \Gamma_N}(\lambda) (\psi) = \left( \partial_\nu u \right)_{|\Gamma_N},
\end{equation}
where $u$ is the unique solution of (\ref{Eq0-Schrodinger}) and $\left( \partial_\nu u \right)_{|\Gamma_N}$ denotes as usual normal derivative of $u$ with respect to the unit outer normal vector $\nu$ on $\Gamma_N$. We assume in analogy with \textbf{(Q2)} that $g$ is a fixed background metric. The Calder\'on problem is now to determine the unknown potential $V \in L^\infty(M)$ from the knowledge of the DN map $\Lambda_{g, V, \,\Gamma_D, \Gamma_N}(\lambda)$: \\

\noindent \textbf{(Q3)}: \emph{Let $(M,g)$ be a smooth compact connected Riemannian manifold with smooth boundary $\partial M$ and let $\Gamma_D, \Gamma_N$ be open subsets of $\partial M$. Let $V_1$ and $V_2$ be potentials in $L^\infty(M)$ and assume that  $\lambda \in \R$ does not belong to the Dirichlet spectra of $-\triangle_g + V_1$ and $-\triangle_g + V_2$. Suppose that
$$
  \Lambda_{g, V_1, \Gamma_D, \Gamma_N}(\lambda) = \Lambda_{g, V_2, \Gamma_D, \Gamma_N}(\lambda).
$$
Does this imply that}
$$
  V_1 = V_2?
$$

As mentioned above, there is a close connection between \textbf{(Q2)} and \textbf{(Q3)} when $\dim M \geq 3$, which is induced by the transformation law for the Laplace-Beltrami operator under conformal changes of metric, that is,
\begin{equation} \label{ConformalScaling}
  -\Delta_{c^4 g} u = c^{-(n+2)} \left( -\Delta_g + q_{g,c} \right) \left( c^{n-2} u \right),
\end{equation}
where
\begin{equation} \label{q}
  q_{g,c} = c^{-n+2} \Delta_{g} c^{n-2}.
\end{equation}	
Indeed, we have:

\begin{prop} \label{Link-c-to-V}
Let $\lambda \in \R$ be fixed. Assume that $c$ is a smooth strictly positive function on $M$ such that $c = 1$ on $\Gamma_D \cup \Gamma_N$. \\
1. If $\Gamma_D \cap \Gamma_N = \emptyset$, then
\begin{equation} \label{Link}
	\Lambda_{c^4 g, \Gamma_D, \Gamma_N}(\lambda) = \Lambda_{g, V_{g,c,\lambda}, \Gamma_D,\Gamma_N}(\lambda),
\end{equation}
where
\begin{equation} \label{Vgc}
  V_{g,c,\lambda} = q_{g,c} + \lambda(1-c^4), \quad q_{g,c} = c^{-n+2} \Delta_{g} c^{n-2}.
\end{equation}	
2. If $\Gamma_D \cap \Gamma_N \ne \emptyset$ and $\partial_{\nu} c = 0$ on $\Gamma_N$, then (\ref{Link}) also holds.
\end{prop}
We refer to \cite{DKN3} for the proof of this result.
$$
$$

As an application of the above result, one can show that \textbf{(Q3)} implies \textbf{(Q2)} in the case of local data, meaning that $\Gamma_D = \Gamma_N = \Gamma$, where $\Gamma$ is an arbitrary open subset in $\partial M$. We now state this result, the proof of which is again given in \cite{DKN3}:

\begin{prop} \label{Q3-to-Q2}
  If $\Gamma_D = \Gamma_N = \Gamma$ is any open set in $\partial M$ and $\lambda \in \R$, then \textbf{(Q3)} implies \textbf{(Q2)}.
\end{prop}



In the remainder of this introduction, we give a brief survey of some of the most important known results on the Calder\'on conjecture. We first remark that the most complete results known for Problems \textbf{(Q1)}, \textbf{(Q2)} and \textbf{(Q3)} apply to the case of \emph{zero frequency} $\lambda = 0$, assuming full data, that is $\Gamma_D = \Gamma_N = \partial M$, or local data, meaning $\Gamma_D = \Gamma_N = \Gamma$ with $\Gamma$ any open subset of $M$. In the particular case of dimension $2$, the anisotropic Calder\'on problem \textbf{(Q1)} for global and local data with $\lambda = 0$ has been given a positive answer for the case of compact connected surfaces in \cite{LaU, LeU}. We also refer to \cite{ALP} for results of a similar nature on \textbf{(Q1)} on bounded domains of $\R^n$, for global and local data, under the weaker regularity hypothesis that the metric is only $L^\infty$. A positive answer to \textbf{(Q1)} for global and local data and zero frequency $\lambda = 0$ in dimension $3$ or higher has been given in \cite{LeU} assuming that the underlying Riemannian manifold is real analytic, compact and connected, with real analytic boundary, and that it further satisfies certain specific topological assumptions. These assumptions were later weakened in \cite{LaU, LaTU}. Similarly, \textbf{(Q1)} has been answered positively for compact connected Einstein manifolds with boundary in \cite{GSB}.

If we don't assume the analyticity of the underlying metrics, the general anisotropic Calder\'on problem \textbf{(Q1)} in dimension $n\geq 3$  is still a major open problem, whether one is dealing with the case of full or local data. Some important results have however been obtained recently on \textbf{(Q2)} and \textbf{(Q3)} in \cite{DSFKSU, DSFKLS, KS1}, for special classes of smooth compact connected Riemannian manifolds with boundary which are referred to as \emph{admissible}. By definition, admissible manifolds $(M,g)$ are \emph{conformally transversally anisotropic},
$$
  M \subset \subset \R \times M_0, \quad g = c ( e \oplus g_0),
$$
where $(M_0,g_0)$ is an $n-1$ dimensional smooth compact connected Riemannian manifold with boundary, $e$ is the Euclidean metric on the real line and $c$ is a smooth strictly positive function in the cylinder $\R \times M_0$. Furthermore the geodesic ray transform on the transversal manifold $(M_0, g_0)$ is assumed to be injective. This is the case for instance if the transversal manifold is \emph{simple}, meaning that any two points in $M_0$ can be connected by a unique geodesic depending smoothly on the endpoints, and $\partial M_0$ is strictly convex as a submanifold of $(M,g) = c ( e \oplus g_0)$. It has been shown in \cite{DSFKSU, DSFKLS} that for admissible manifolds, the conformal factor $c$ is uniquely determined from the knowledge of the DN map at zero frequency, so that both \textbf{(Q2)} and \textbf{(Q3)} have positive answers. These results have been further extended to the case of partial data in \cite{KS1}. We refer to \cite{GT1, Is, IUY1} for additional results in the  case of local data and to the surveys \cite{GT2, KS2} for further references.

For bounded domains $\Omega$ of $\R^n, \ n \geq 3$ endowed with the Euclidean metric, there are also positive results  for problem \textbf{(Q3)}, for data measured on distinct subsets $\Gamma_D, \Gamma_N$ of $\partial M$ which are not assumed to be disjoint, \cite{KSU}. The hypothesis is that the sets $\Gamma_D, \Gamma_N$ should overlap, allowing however for $\Gamma_D \subset \partial \Omega$ to possibly have very small measure, and requiring then that $\Gamma_N$ have slightly larger measure than $\partial \Omega \setminus \Gamma_D$. These results have been generalized in \cite{KS1} to the case of admissible Riemannian manifolds, using limiting Carleman weights $\varphi$,
which make it possible to decompose the boundary of $M$ as
$$
  \partial M = \partial M_+ \cup \partial M_{\textrm{tan}} \cup \partial M_-,
$$
where
$$
  \partial M_\pm = \{ x \in \partial M: \ \pm \partial_\nu \varphi(x) > 0 \}, \quad \partial M_{\textrm{tan}} = \{ x \in \partial M: \  \partial_\nu \varphi(x) = 0 \}.
$$
Under additional geometric assumptions on the transverse manifold $(M_0,g_0)$, it is shown in \cite{KS1} that the answer to \textbf{(Q3)} is positive if $\Gamma_D$ contains $\partial M_- \cup \Gamma_a$ and $\Gamma_N$ contains $\partial M_+ \cup \Gamma_a$, where $\Gamma_a$ is some open subset of $\partial M_{\textrm{tan}}$. This implies that $\Gamma_D$ and $\Gamma_N$ must overlap in order to have uniqueness in this setting. The only exception occurs in the case where $\partial M_{\textrm{tan}}$ has zero measure, in which case it is enough to take $\Gamma_D = \partial M_-$ and $\Gamma_N = \partial M_+$ to have uniqueness for \textbf{(Q3)} (see Theorem 2.3 in  \cite{KS1}). Note in this case that $\Gamma_D \cap \Gamma_N = \partial M_- \cap \partial M_+ = \emptyset$.

In the case  of data measured on \emph{disjoint sets}, the only known results prior to \cite{DKN2, DKN3} appear to be those of \cite{KS1}, which hold for the case of zero frequency $\lambda = 0$ and concern classes of admissible Riemannian manifolds, and those of  \cite{IUY2} which apply to the case of a potential for a Schr\"odinger operator on a two-dimensional domain homeomorphic to a disc. For example in the latter work, it is shown that when the boundary of the domain is partitioned into eight clockwise-ordered arcs $\Gamma_1, \Gamma_2, \dots, \Gamma_8$, then the potential is determined when the Dirichlet data are supported on $S = \Gamma_2 \cup \Gamma_6$ and the Neumann data are observed on $R = \Gamma_4 \cup \Gamma_8$, hence answering \textbf{(Q3)} positively in this special setting.

Finally, we mention some related papers concerned with the \emph{hyperbolic} anisotropic Calder\'on problem, which is the case in which the partial DN map is assumed to be known at all frequencies $\lambda$, see \cite{Rak, LO1, LO2, KLO}. We refer to \cite{KKL} for a detailed discussion of the hyperbolic anisotropic Calder\'on problem and to \cite{KKLM} for the link between the hyperbolic DN map and the elliptic DN map at all frequencies.

The rest of our paper is organized as follows. In Section \ref{1}, we recall from  \cite{DKN3} the definition of a new type of gauge invariance for the anisotropic Calder\'on problem with data on disjoint sets. This new gauge invariance corresponds to special rescalings of the fixed background metric $g$ by a conformal factor which solves a suitably chosen boundary value problem for a nonlinear elliptic PDE of Yamabe type. Section \ref{Countergen} is devoted to the description of the counterexamples to uniqueness for the anisotropic Calder\'on problem modulo this new gauge invariance. These take the form of Schr\"odinger operators on cylindrical warped products of dimension $n \geq 2$, or conformal rescalings of cylindrical warped products of dimension $n \geq 3$. The paper concludes with some remarks on the case of manifolds with corners.

\section{A new gauge invariance for the Calder\'on problem with disjoint data} \label{1}

We now describe a new kind of gauge invariance for the Calder\'on problem, which was first introduced in \cite{DKN3} following earlier work \cite{DKN2} in which we showed through explicit counterexamples that the answers to \textbf{(Q2)} (and thus \textbf{(Q1)}) as well as \textbf{(Q3)} were \emph{negative} when the Dirichlet and Neumann data are measured on disjoint sets of the boundary. These examples take the form of special rotationally invariant toric cylinders of dimensions $2$ and $3$ . More precisely, we constructed in \cite{DKN2} an infinite number of pairs of non isometric metrics and potentials having the same partial DN maps when $\Gamma_D \cap \Gamma_N = \emptyset$ and for any fixed frequency $\lambda$ not belonging to the Dirichlet spectra of the corresponding Laplace-Beltrami or Schr\"odinger operators. It is a particularly noteworthy feature of this construction that any pair of such metrics belongs to the same conformal class, with the conformal factor relating the two metrics satisfying a specific nonlinear ODE. We subsequently showed in \cite{DKN3} that the mechanism underlying the non-uniqueness results of \cite{DKN2} can be broadly generalized to provide counterexamples to uniqueness for the anisotropic Calder\'on problem for any smooth compact connected Riemannian manifold with smooth boundary, of dimension three or higher, with Dirichlet data and Neumann data given on disjoint subsets $\Gamma_D$ and $\Gamma_N$ such that $\overline{\Gamma_D \cup \Gamma_N} \ne \partial M$.
These counterexamples are also closely tied to rescalings of a fixed metric $g$ by a conformal factor, which now satisfies a nonlinear elliptic PDE of Yamabe type with appropriately chosen boundary conditions instead of a nonlinear ODE (see Theorem \ref{Main-1}). The proof of the existence of smooth positive solutions of this nonlinear equation is achieved using the standard technique of lower and upper solutions. We emphasize that this technique works thanks to the crucial assumption $\overline{\Gamma_D \cup \Gamma_N} \ne \partial M$, that allows us to choose appropriately the boundary conditions appearing in the nonlinear equation. We now recall these results from \cite{DKN3} by first presenting in the form of a proposition the elliptic boundary value problem of Yamabe type that is at the basis of this additional and somewhat hidden gauge invariance. We have:

\begin{prop} \label{Main}
Let $(M,g)$ be a smooth compact connected Riemannian manifold of dimension $n\geq 3$ with smooth boundary $\partial M$ and let $\lambda \in \R$ not belong to the Dirichlet spectrum $\sigma(-\Delta_g)$. Let $\Gamma_D, \Gamma_N$ be open sets of $\partial M$ such that $\Gamma_D \cap \Gamma_N = \emptyset$. If there exists a smooth strictly positive function $c$ satisfying
\begin{equation} \label{Main-EDP}
  \left\{ \begin{array}{cc} \Delta_{g} c^{n-2} + \lambda ( c^{n-2} - c^{n+2}) = 0, & \textrm{on} \ M, \\
	c = 1, & \textrm{on} \ \Gamma_D \cup \Gamma_N, \end{array} \right.
\end{equation}
then the conformally rescaled Riemannian metric $\tilde{g} = c^4 g$ satisfies
$$
  \Lambda_{\tilde{g},\Gamma_D, \Gamma_N}(\lambda) = \Lambda_{g,\Gamma_D, \Gamma_N}(\lambda).  	
$$
\end{prop}
We refer to \cite{DKN3} for a proof of the above proposition.
We also note that the nonlinear PDE (\ref{Main-EDP}) satisfied by the conformal factor $c$ may be re-expressed in more geometric terms by making use of the well-known fact that the potential $q_{g,c}$ in (\ref{q}) can be written as
	\begin{equation} \label{ScalarCurvature}
	  q_{g,c} = \frac{n-2}{4(n-1)} \left( Scal_g - c^4 \, Scal_{c^4 g} \right),
	\end{equation}
	where $Scal_g$ and $Scal_{c^4 g}$ denote the scalar curvatures associated to $g$ and $\tilde{g} = c^4 g$ respectively. Indeed, it is easily seen that $c$ will satisfy (\ref{Main-EDP}) is and only if
\begin{equation} \label{GeometricInterpretation}
  Scal_{c^4 g} = \frac{Scal_g + \frac{4(n-1)}{n-2} \lambda (1-c^4)}{c^4}.
\end{equation}

It follows from Proposition \ref{Main} that in order to construct counterexamples to uniqueness for the Calder\'on problem in dimension $n\geq 3$ with data on disjoint subsets of the boundary, it is sufficient to find a conformal factor $c$ satisfying the nonlinear PDE of Yamabe type (\ref{Main-EDP}), such that $c \ne 1$ on $M$ (see \ref{Inv-Conformal-1}). This can been done by using the well known technique of lower and upper solutions. Indeed, we are interested in solutions $w = c^{n-2}$ of the nonlinear elliptic PDE:
\begin{equation} \label{Eqw}
  \left\{ \begin{array}{cc} \Delta_g w + f(w) =0 , & \textrm{on} \ M, \\ w = \eta, & \textrm{on} \ \partial M, \end{array} \right.
\end{equation}
where $f(w) = \lambda (w-w^{\frac{n+2}{n-2}})$ and $\eta$ is a smooth function on $\partial M$ such that $\eta = 1$ on $\Gamma_D \cup \Gamma_N$. We may even more generally consider the nonlinear Dirichlet problem
\begin{equation} \label{GeneralDP}
  \left\{ \begin{array}{cc} \Delta_g w + f(x,w) =0 , & \textrm{on} \ M, \\ w = \eta, & \textrm{on} \ \partial M, \end{array} \right.
\end{equation}
where $f$ is a smooth function on $M \times \R$ and $\eta$ is a smooth function on $\partial M$.


If we can find a lower solution ${\underline{w}}$ and an upper solution ${\overline{w}}$ satisfying ${\underline{w}} \leq {\overline{w}}$ on $M$, then there exists a solution $w \in C^{\infty}(\overline{M})$ of (\ref{GeneralDP}) such that ${\underline{w}} \leq w \leq {\overline{w}}$ on $M$ (see for example \cite{Sat}, Thm 2.3.1. or \cite{Ta2}, Section 14.1). We briefly recall from \cite{DKN3} the construction of such a solution : we pick $\mu>0$ such that $|\partial_w f(x,w)| \leq \mu$ for $w \in [\min \  {\underline{w}} , \max \ {\overline{w}}]$.  Then, we define recursively a sequence $(w_k)$ by $w_0 = {\underline{w}}$, $w_{k+1} = \Phi(w_k)$ where $\Phi(w) = \varphi$ is obtained by solving
\begin{equation}
\Delta_g \varphi - \mu \varphi = -\mu w - f(x,w) \ ,\ \varphi_{|\partial M} = \eta.
\end{equation}
Using the maximum principle, we see that this sequence satisfies
\begin{equation}\label{sequence}
{\underline{w}}=w_0 \leq w_1 \leq \cdots \leq w_k \cdots \leq {\overline{w}}.
\end{equation}
We therefore deduce that $w = \displaystyle\lim_{k \to \infty} w_k$ is a solution of (\ref{Eqw}). The details of the construction are given in the above references \cite{Sat, Ta2}.

We thus obtain the following elementary result, the proof of which is given in \cite{DKN3}:

\begin{prop} \label{NonlinearDirichletPb}
For all $\lambda \geq 0$, (resp. for all $\lambda < 0$), and for all smooth positive functions $\eta$ such that $\eta \ne 1$ on $\partial M$, (resp. $\eta \lneq 1$ on $\partial M$),  there exists a positive solution $w \in C^{\infty}(\overline{M})$ of (\ref{Eqw}) satisfying $w \ne 1$ on $M$.
\end{prop}

In order to use the existence results of Proposition \ref{NonlinearDirichletPb} for the construction of a conformal factor $c$ satisfying (\ref{Main-EDP}) and $c \ne 1$ on $M$, we need to be able to choose $\eta \ne 1$ on $\partial M$. We thus make the crucial assumption on the disjoint Dirichlet and Neumann data that
\begin{equation} \label{Main-Hyp}
  \overline{\Gamma_D \cup \Gamma_N} \ne \partial M.
\end{equation}
Putting together then the results of Proposition \ref{Main} and Proposition \ref{NonlinearDirichletPb}, we obtain:

\begin{thm} \label{Main-1}
  Let $(M,g)$ be a smooth compact connected Riemannian manifold of dimension $n\geq 3$ with smooth boundary $\partial M$. Let $\Gamma_D, \Gamma_N$ be open subsets of $\partial M$ such that $\Gamma_D \cap \Gamma_N = \emptyset$ and $\overline{\Gamma_D \cup \Gamma_N} \ne \partial M$. Consider a conformal factor $c \ne 1$ on $M$ whose existence is given in Proposition \ref{NonlinearDirichletPb}, defined as a smooth solution of the nonlinear Dirichlet problem
\begin{equation} \label{Main-EDP-1}
  \left\{ \begin{array}{cc} \Delta_{g} c^{n-2} + \lambda (c^{n-2} - c^{n+2}) = 0, & \textrm{on} \ M, \\
	c^{n-2} = \eta, & \textrm{on} \ \partial M, \end{array} \right.
\end{equation}
where $\eta$ is a smooth positive function on $\partial M$ satisfying $\eta = 1$ on $\Gamma_D \cup \Gamma_N$ and $\eta \ne 1$ on $\partial M \setminus (\Gamma_D \cup \Gamma_N)$. Then the Riemannian metric $\tilde{g} = c^4 g$ with $c \ne 1$ on $M$ satisfies
$$
  \Lambda_{\tilde{g},\Gamma_D, \Gamma_N}(\lambda) = \Lambda_{g,\Gamma_D, \Gamma_N}(\lambda).  	
$$
\end{thm}

We interpret the content of Theorem \ref{Main-1} as defining a new gauge invariance for the anisotropic Calder\'on problem with disjoint data. This definition is formalized in the following way:

\begin{defi}[New gauge invariance] \label{Gauge0}
  Let $(M,g)$ and $(M,\tilde{g})$ be smooth compact connected Riemannian manifolds of dimension $n\geq 3$ with smooth boundary $\partial M$. Let $\lambda \in \R$  not belong to the union of the Dirichlet spectra of $-\Delta_g$ and $-\Delta_{\tilde{g}}$. Let $\Gamma_D, \Gamma_N$ be open subsets of $\partial M$ such that $\Gamma_D \cap \Gamma_N = \emptyset$ and $\overline{\Gamma_D \cup \Gamma_N} \ne \partial M$. We say that $g$ and $\tilde{g}$ are gauge related if there exists a smooth positive conformal factor $c$ such that: \\
\begin{equation} \label{Gauge}
   \left\{ \begin{array}{rl} \tilde{g} & = c^4 g, \\
	                           \Delta_{g} c^{n-2} + \lambda (c^{n-2} - c^{n+2}) & = 0, \textrm{on} \ M, \\
	                           c & = 1, \textrm{on} \ \Gamma_D \cup \Gamma_N, \\
														 c & \ne 1, \textrm{on} \ \partial M \setminus (\Gamma_D \cup \Gamma_N).
														\end{array} \right.
   \end{equation}
In that case, we have: $\Lambda_{\tilde{g}, \Gamma_D, \Gamma_N}(\lambda) = \Lambda_{g, \Gamma_D, \Gamma_N}(\lambda)$.
\end{defi}

\begin{rem}
  In dimension $2$, the gauge invariance introduced in Definition \ref{Gauge0} for the anisotropic Calder\'on problem with disjoint data is not relevant except for the case of zero frequency. Indeed, the nonlinear PDE (\ref{Gauge}) that the conformal factor $c$ should satisfy becomes
\begin{equation} \label{EDP-Dim2}
  \lambda (1 - c^4) = 0, \ \textrm{on} \ M.
\end{equation}
In other words, $c$ must be identically equal to $1$ if $\lambda \ne 0$. Recalling that in dimension $2$ and for zero frequency, a conformal transformation is already known to be a gauge invariance of the anisotropic Calder\'on problem, we see that our construction will not lead to new counterexamples to uniqueness in dimension $2$, for any frequency $\lambda$.
\end{rem}

We conclude this Section by stating a version of the anisotropic Calder\'on conjecture with disjoint data modulo the previously defined gauge invariance. \\

\noindent \textbf{(Q4)} \emph{Let $M$ be a smooth compact connected manifold with smooth boundary $\partial M$ and let $g,\, \tilde{g}$ be smooth Riemannian metrics on $M$. Let $\Gamma_D, \Gamma_N$ be any open sets of $\partial M$ such that $\Gamma_D \cap \Gamma_N = \emptyset$ and $\lambda \in \R$ not belong to $\sigma(-\Delta_g) \cup \sigma(-\Delta_{\tilde{g}})$. If $\Lambda_{g,\Gamma_D, \Gamma_N}(\lambda) = \Lambda_{\tilde{g},\Gamma_D, \Gamma_N}(\lambda)$, is it true that $g = \tilde{g}$ up to the following gauge invariances: \\
1. (\ref{Inv-Diff}) in any dimension, \\
2. (\ref{Inv-Conf}) if $\dim M = 2$ and $\lambda = 0$}, \\
3. (\ref{Gauge}) if $\dim M \geq 3$ and $\overline{\Gamma_D \cup \Gamma_N} \ne \partial M$?


\Section{Counterexamples to uniqueness for the anisotropic Calder\'on problem with disjoint data modulo the gauge invariance}\label{Countergen}

\subsection{The case of Schr\"odinger operators on cylindrical warped products in dimension $n\geq 3$} \label{2}

In this subsection, we consider the anisotropic Calder\'on problem \textbf{(Q3)} for Schr\"odinger operators on a fixed smooth compact connected Riemannian manifold $(M,g)$ of dimension $n\geq 2$, with smooth boundary $\partial M$, under the assumption that the Dirichlet and Neumann data are measured on disjoint subsets of the boundary. In view of the link (\ref{Link}) between the Calder\'on problems \textbf{(Q2)} and \textbf{(Q3)}, one might think that the previously constructed counterexamples to uniqueness for the anisotropic Calder\'on problem \textbf{(Q2)} in dimension $3$ or higher could be used to construct counterexamples to uniqueness for \textbf{(Q3)}. It turns out that this is not the case. To this effect, we recall first the following lemma from \cite{DKN3}:

\begin{lemma}\label{lemmafactor}
Let $(M,g)$ be a smooth compact connected Riemannian manifold of dimension $n\geq 3$ with smooth boundary $\partial M$. Consider two smooth conformal factors $c_1$ and $c_2$ such that $c:= \frac{c_2}{c_1}$ satisfies
\begin{equation}\label{factor}
\Delta_{c_1^4 g} c^{n-2} + \lambda (c^{n-2} - c^{n+2}) = 0 \ \rm{on}\  M.
\end{equation}
Then,
\begin{equation} \label{c2}
  V_{g,c_1,\lambda} = V_{g,c_2,\lambda}.
\end{equation}
\end{lemma}


Now, let $\Gamma_D, \Gamma_N$ be open subsets of $\partial M$ such that $\Gamma_D \cap \Gamma_N = \emptyset$ and $\overline{\Gamma_D \cup \Gamma_N} \ne \partial M$. Consider two smooth conformal factors $c_1$ and $c_2$ such that the metrics $G=c_1^4 g $ and $\tilde{G}=c_2^4 g$ are gauge equivalent in the sense of Definition \ref{Gauge0}. This implies in particular that $\left(\frac{c_2}{c_1}\right)^4$ satisfies (\ref{factor}) and that $\Lambda_{G, \Gamma_D, \Gamma_N}(\lambda) = \Lambda_{\tilde{G}, \Gamma_D, \Gamma_N}(\lambda)$. We obtain from (\ref{Link}) that
$$
  \Lambda_{g, V_{g,c_1,_\lambda}, \Gamma_D, \Gamma_N}(\lambda) = \Lambda_{g, V_{g,c_2,\lambda}, \Gamma_D, \Gamma_N}(\lambda).
$$
But Lemma \ref{lemmafactor} implies in turn that  $ V_{g,c_1,\lambda} = V_{g,c_2,\lambda}$. As a consequence, the gauge invariance for the anisotropic Calder\'on problem \textbf{(Q2)} with disjoint data highlighted in Section \ref{1} is not a gauge invariance for the corresponding anisotropic Calder\'on problem \textbf{(Q3)}.

In \cite{DKN2, DKN3}, we found however counterexamples to uniqueness for the anisotropic Calder\'on problem \textbf{(Q3)} with disjoint sets on a specific class of \emph{smooth} compact connected Riemannian cylinders equipped with a \emph{warped product} metric and having \emph{two ends}, \textit{i.e.} whose boundary has two distinct connected components. The warped product structure is crucial here since it allows separation of variables with respect to one variable. In particular, imposing that the potentials in the Schr\"odinger equation only depend on the euclidean direction of the cylinder, the global Dirichlet problem reduces to a countable family of ODEs in the separated variable, parametrized by the angular momenta arising from the diagonalization of the Laplacian on the transverse manifold. As a consequence, the global DN map can be decomposed into a direct sum of one-dimensional (partial) DN maps corresponding to each of the above ODEs and powerful 1D inverse spectral techniques can be used to study them. Even more important for the construction of our counterexamples is the fact that the smooth cylinder has two ends. Indeed, it will be shown below - through an explicit construction - that the information contained in the partial DN maps radically differs according to whether we measure the Dirichlet and Neumann data (even disjoint) on a same connected component of the boundary, or if we measure them on two distinct connected components. In the latter case, the information contained in the partial DN maps will be shown to be insufficient to conclude to uniqueness. Finally, we mention that if we allow manifolds which are not smooth, \textit{e.g.} manifolds with corners, we can remove the assumption on the non-connectedness of the boundary (see section \ref{4} below).

Let us be more explicit and recall here the construction of these counterexamples. We consider cylinders $M = [0,1]\times K$, where $K$ is an arbitrary $(n-1)$-dimensional closed manifold, equipped with a Riemannian metric of the form
\begin{equation} \label{Metric}
  g = f^4(x) [dx^2 + g_K],
\end{equation}
where $f$ is a smooth strictly positive function on $[0,1]$ and $g_K$ denotes a smooth Riemannian metric on $K$. Clearly, $(M,g)$ is a $n$-dimensional warped product cylinder and the boundary $\partial M$ has two connected components, that is $\partial M = \Gamma_0 \cup \Gamma_1$ where $\Gamma_0 = \{0\} \times K$ and $\Gamma_1 = \{1\} \times K$ correspond to the two ends of $(M,g)$. Let $-\triangle_g$ denote the positive Laplace-Beltrami operator on $(M,g)$ and consider a potential $V = V(x) \in L^\infty(M)$ (or $L^2(M)$) and $\lambda \in \R$ such that $\lambda \notin \{ \lambda_j\}_{j \geq 1}$ where $\{ \lambda_j\}_{j \geq 1}$ is the Dirichlet spectrum of $-\Delta_g + V$. Given Dirichlet and Neumann data $\Gamma_D, \Gamma_N$ on $\partial M$, let us define the DN map $\Lambda_{g,V,\Gamma_D,\Gamma_N}(\lambda)$ as in (\ref{DN-Abstract-Schrodinger}).

We first construct the global DN map $\Lambda_{g,V}(\lambda)$ and then obtain $\Lambda_{g,V,\Gamma_D,\Gamma_N}(\lambda)$ by restricting the Dirichlet and Neumann data to $\Gamma_D$ and $\Gamma_N$. The boundary $\partial M$ of $M$ having two disjoint components $\partial M = \Gamma_0 \cup \Gamma_1$, let us decompose the Sobolev spaces $H^s(\partial M)$ as $H^s(\partial M) = H^s(\Gamma_0) \oplus H^s(\Gamma_1)$ for any $s \in \R$ and use the vector notation
$$
 \varphi = \left( \begin{array}{c} \varphi^0 \\ \varphi^1 \end{array} \right),
$$
to denote the elements $\varphi$ of $H^s(\partial M) = H^s(\Gamma_0) \oplus H^s(\Gamma_1)$. Since the DN map is a linear operator from $H^{1/2}(\partial M)$ to $H^{-1/2}(\partial M)$, it has the structure of an operator valued $2 \times 2$ matrix
\begin{equation} \label{h1}
  \Lambda_g(\lambda) = \left( \begin{array}{cc} \Lambda_{g,\Gamma_0,\Gamma_0}(\lambda) & \Lambda_{g,\Gamma_1, \Gamma_0}(\lambda) \\ \Lambda_{g,\Gamma_0, \Gamma_1}(\lambda) & \Lambda_{g,\Gamma_1,\Gamma_1}(\lambda) \end{array} \right),
\end{equation}
whose components are operators from $H^{1/2}(K)$ to $H^{-1/2}(K)$.

Let us now use the warped product structure of $(M,g)$ to simplify the expression of $\Lambda_{g,V}(\lambda)$. First, setting $v = f^{n-2} u$, the Dirichlet problem (\ref{Eq0-Schrodinger}) can be written as (see \cite{DKN3})

\begin{equation} \label{Eq1}
  \left\{ \begin{array}{rcl}
	\left[ -\partial^2_x - \triangle_K + q_f + (V-\lambda) f^4 \right] v & = & 0, \ \textrm{on} \ M, \\
	v & = & f^{n-2} \psi, \ \textrm{on} \ \partial M,
	\end{array} \right.
\end{equation}
where $-\triangle_K$ denotes the positive Laplace-Beltrami operator on $(K,g_K)$ and $q_f = \frac{(f^{n-2})''}{f^{n-2}}$.

Second, we introduce the Hilbert basis consisting of the harmonics $(Y_k)_{k \geq 0}$ of the Laplace-Beltrami operator $-\triangle_K$, \textit{i.e.} the $Y_k$'s are the normalized eigenfunctions of $-\triangle_K$ associated to the eigenvalues $\mu_k$ ordered (counting multiplicity) such that
$$
  0 = \mu_0 < \mu_1 \leq \mu_2 \leq \dots \leq \mu_k \to \infty.
$$
Clearly, we can separate variables in the equation (\ref{Eq1}) by looking for the unique solution $v$ in the form
$$
  v = \sum_{k \geq 0} v_k(x) Y_k(\omega).
$$
The functions $v_k$ satisfy then the 1D boundary value problems
\begin{equation} \label{Eq2}
  \left\{ \begin{array}{c} -v_k'' + [ q_f + (V-\lambda) f^4] v_k = -\mu_k v_k, \ \textrm{on} \ [0,1], \\
   v_k(0) = f^{n-2}(0) \psi^0_k, \quad v_k(1) = f^{n-2}(1) \psi^1_k, \end{array} \right.
\end{equation}
where we wrote the Dirichlet data $\psi = (\psi^0, \psi^1) \in H^{1/2}(\partial M)$ using their Fourier expansions as $\psi^0 = \sum_{k \geq 0} \psi^0_k Y_k$, $\psi^1 = \sum_{k \geq 0} \psi^1_k Y_k$.

It is also clear form the above decomposition that the DN map can be diagonalized in the Hilbert basis $\{ Y_k \}_{k \geq 0}$. Precisely, it was shown in \cite{DKN2, DKN3} that on each Hilbert space $<Y_k>$ spanned by the harmonic $Y_k$, the DN map acts as a multiplication operator by a $2 \times 2$ matrix given explicitly by
\begin{equation} \label{DN-Partiel}
  \Lambda_{g,V}(\lambda)_{|<Y_k>} := \Lambda^k_g(\lambda) = \left( \begin{array}{cc} \frac{(n-2)f'(0)}{f^3(0)} - \frac{M_{g,V}(\mu_k)}{f^2(0)} & -\frac{f^{n-2}(1)}{f^n(0) \Delta_{g,V}(\mu_k)} \\ -\frac{f^{n-2}(0)}{f^n(1) \Delta_{g,V}(\mu_k)} & -\frac{(n-2)f'(1)}{f^3(1)} - \frac{N_{g,V}(\mu_k)}{f^2(1)} \end{array} \right) .
\end{equation}

Here the quantities $\Delta_{g,V}(\mu_k)$, $M_{g,V}(\mu_k)$ and $N_{g,V}(\mu_k)$ denote the characteristic and Weyl-Titchmarsh functions of the boundary value problem
\begin{equation} \label{Eq3}
  \left\{ \begin{array}{c} -v'' + [q_{f}(x) +(V-\lambda) f^4(x)] v = - \mu v, \\
  v(0) = 0, \quad v(1) = 0.  \end{array}\right.
\end{equation}
They are defined in the following way. The potential $q_f +(V-\lambda) f^4$ being real and in $L^\infty([0,1])$ or $L^2([0,1])$, we can define for all $\mu \in \C$ two fundamental systems of solutions (FSS) of (\ref{Eq3})
$$
  \{ c_0(x,\mu), s_0(x,\mu)\}, \quad \{ c_1(x,\mu), s_1(x,\mu)\},
$$
by imposing the Cauchy conditions
\begin{equation} \label{FSS}
  \left\{ \begin{array}{cccc} c_0(0,\mu) = 1, & c_0'(0,\mu) = 0, & s_0(0,\mu) = 0, & s_0'(0,\mu) = 1, \\
 	                  c_1(1,\mu) = 1, & c'_1(1,\mu) = 0, & s_1(1,\mu) = 0, & s'_1(1,\mu) = 1. \end{array} \right.
\end{equation}
We recall the following two important properties of the two FSS $\{c_0, s_0\}$ and $\{c_1, s_1\}$.
\begin{enumerate}
\item In terms of the Wronskian $W(u,v) = uv' - u'v$, we have $W(c_0,s_0) = 1, \quad W(c_1,s_1) = 1$.
\item The functions $\mu \to c_j(x,\mu), \, s_j(x,\mu)$ and their derivatives with respect to $x$ are entire functions of order $\frac{1}{2}$ (see \cite{PT}).
\end{enumerate}
The characteristic function of (\ref{Eq3}) is then defined by
\begin{equation} \label{Char}
  \Delta_{g,V}(\mu) = W(s_0, s_1),
\end{equation}
and the Weyl-Titchmarsh functions are defined by
\begin{equation} \label{WT}
  M_{g,V}(\mu) = - \frac{W(c_0, s_1)}{\Delta_{g,V}(\mu)} = -\frac{D_{g,V}(\mu_k)}{\Delta_{g,V}(\mu)}, \quad N_{g,V}(\mu) = - \frac{W(c_1, s_0)}{\Delta_{g,V}(\mu)} = \frac{E_{g,V}(\mu_k)}{\Delta_{g,V}(\mu)}.
\end{equation}
These functions whose relevance in 1D inverse spectral problems is well known (see for instance \cite{Be, Bo1, Bo2, ET, FY, GS, KST, PT}), have the following fundamental properties:

\begin{itemize}
\item The zeros $(\alpha_n)_{n \geq 1}$ of the characteristic function $\Delta_{g,V}$ correspond to minus the Dirichlet spectrum of the selfadjoint Schr\"odinger operator $H = -\frac{d^2}{dx^2} + [q_{f}(x) +(V-\lambda) f^4(x)]$. Moreover, $\Delta_{g,V}$ is completely determined (up to a constant) by the sequence $(\alpha_n)$ through the formula (which is simply a consequence of the Hadamard factorization Theorem):
$$
  \Delta_{g,V}(\mu) = C \prod_{n \geq 1} \left( 1 - \frac{\mu}{\alpha_n} \right).
$$
\item The Weyl-Titchmarsh functions $M_{g,V}(\mu)$ and $N_{g,V}(\mu)$ are meromorphic functions on $\C$ with poles at $(\alpha_n)$. These functions determine uniquely the potential $q_{f}(x) +(V-\lambda) f^4(x)$ through the Borg-Marchenko Theorem which can be stated as follows in our setting : \emph{Assume that $M_{g,V}(\mu) = M_{g,\tilde{V}}(\mu), \ \forall \mu \in \C$. Then $V(x) = \tilde{V}(x), \ \forall x \in [0,1]$. The result is the same if we replace $M_{g,V}(\mu)$ by $N_{g,V}(\mu)$}.
\end{itemize}


Thanks to the expression (\ref{DN-Partiel}) and the previous properties of the characteristic and Weyl-Titchmarsh functions, we can understand more precisely the difference between the amount of information contained in the DN map according to whether we measure the Dirichlet and Neumann data on the same connected component of the boundary, or on distinct connected components.

In the former case which corresponds to the diagonal components of (\ref{DN-Partiel}), the DN map on each harmonic $Y_k$ is simply an operator of multiplication by an expression containing some boundary values of the metric $g$ and its first normal derivative $\partial_\nu g$ as well as the Weyl-Titchmarsh functions $M_{g,V} $ or $N_{g,V}$ evaluated at the $\{\mu_k\}_{k \geq 0}$. Using this information, we can in general prove uniqueness. For instance, we have:

\begin{thm} \label{Uniqueness}
Let $(M,g)$ be a smooth compact connected warped product cylinder as in (\ref{Metric}). Let $V(x), \tilde{V}(x) \in L^\infty(M)$ or $L^2(M)$. Let $\lambda \in \R$ not belong to the Dirichlet spectra of $-\triangle_g + V$ and $-\triangle_g + \tilde{V}$. Assume that
$$
  \Lambda_{g,V, \Gamma_0, \Gamma_0}(\lambda) = \Lambda_{g,\tilde{V}, \Gamma_0, \Gamma_0}(\lambda).
$$
Then $V(x) = \tilde{V}(x)$ for all $x \in [0,1]$.
\end{thm}

\begin{proof}
Our assumption means that:
$$
  M_{g,V}(\mu_k) = M_{g,\tilde{V}}(\mu_k), \quad \forall k \geq 1.
$$
Using (\ref{WT}), this implies
$$
  D_{g,V}(\mu_k) \tilde{\Delta}_{g,V}(\mu_k) - \tilde{D}_{g,V}(\mu_k) \Delta_{g,V}(\mu_k) = 0, \quad \forall k \geq 1.
$$
Introduce the function $F(\mu) = D_{g,V}(\mu^2) \tilde{\Delta}_{g,V}(\mu^2) - \tilde{D}_{g,V}(\mu^2) \Delta_{g,V}(\mu^2)$. From the analytic properties of the FSS $\{c_0, s_0\}$ and $\{c_1, s_1\}$, we see that $F$ is an entire function of order $1$ that vanishes on the sequence $(\sqrt{\mu_k})$. Moreover $F$ is bounded on the imaginary axis $i\R$. It follows that $F$ belongs to the Nevanlina class \cite{Lev}). Let us show that $F$ must vanish identically on $\C$. First, the Weyl law implies the following asymptotics on the $\sqrt{\mu_k}$ (repeated according multiplicity):
$$
  \sqrt{\mu_k} =  C k^{\frac{1}{(n-1)}} +O(1),
$$
where $C$ denotes a suitable constant independent of $k$. Next, setting  $G(\mu) = F(C \mu)$, we see that $G$ vanishes on the sub-sequence $\lambda_k := 
\frac{1}{C} \sqrt{\mu_{k^{n-1}}}$, which satisfies $\lambda_k = k+O(1)$. But, since the $\lambda_k$ could be counted many times, we still need to have recourse to one further trick in order to conclude. Namely, for $N \in \N$, we introduce a new function $H(\nu)= G(N\nu)$ which, just like  $F(\nu)$, is entire of order $1$ and bounded on the imaginary axis, and which vanishes on $\nu_k = \frac{1}{N} \lambda_{Nk}$. It follows from the previous discussion that $\nu_k \sim k$ and are distinct if $N$ is large enough. Since $\sum \frac{1}{\nu_k} = +\infty$, we conclude thus that $H(\mu) = 0$, (and then $F(\mu)=0$), for all $\mu \in \C$. From the definition of $F$, this result can be translated into the equality between the Weyl-Titchmarsh functions $M_{g,V}(\mu) = M_{g,\tilde{V}}(\mu)$ for all $\mu \in \C$. Applying the Borg-Marchenko Theorem, we finally get $V(x) = \tilde{V}(x)$ for all $x \in [0,1]$.
\end{proof}


\begin{rem}
It is an open problem to prove uniqueness of the potential $V$ from $\Lambda_{g,V,\Gamma_D,\Gamma_N}(\lambda)$ if we measure for instance the Dirichlet and Neumann data $\Gamma_D = \Gamma_N = \Gamma$ on an open subset strictly contained in a connected component of $\partial M$, \textit{i.e}. $\Gamma \subsetneq \Gamma_0$ or $\Gamma \subsetneq \Gamma_1$. For some uniqueness results in that direction in the particular case of rotationally invariant toric cylinders, we refer to \cite{DKN2}.
\end{rem}

Let us come back now to non-uniqueness results. They appear in the case where the Dirichlet and Neumann data are measured on distinct connected components of the boundary. In this case, which corresponds to measuring the anti-diagonal components of (\ref{DN-Partiel}), the situation becomes much more interesting since the Weyl-Titchmarsh functions are replaced by the characteristic functions $\Delta_{g,V}$. As a consequence, the above argument cannot work since there exist no equivalent result of the Borg-Marchenko Theorem from the characteristic function. On the contrary, it is well known that the characteristic function is not enough to determine uniquely a potential. We can make precise this assertion by stating the following key result whose proof almost readily follows from the properties on the characteristic function mentioned above and can be found in \cite{DKN3}:

\begin{lemma} \label{Link-Iso}
  Let $g$ be a fixed metric as in (\ref{Metric}) and $V = V(x), \tilde{V} = \tilde{V}(x) \in L^2(M)$. Then
$$
	\Delta_{g,V}(\mu) = \Delta_{g,\tilde{V}}(\mu), \quad \forall \mu \in \C,
$$
if and only if
$$
  q_f + (V-\lambda)f^4 \ \textrm{and} \ q_f + (\tilde{V}-\lambda)f^4 \ \textrm{are isospectral for} \ (\ref{Eq3}).
$$	
\end{lemma}

In the case of a potential $Q = q_f + (V-\lambda)f^4$ that belongs to $L^2([0,1])$, we have a complete characterization of the class of isospectral potentials to $Q$ for the Schr\"odinger operator with Dirichlet boundary conditions (\ref{Eq3}). This is due to the fundamental work of P\"oschel and Trubowitz~\cite{PT}. In particular, given such a potential $Q$, the family of isospectral potentials is parametrized by sequences $\xi \in l^2_1$ where $l^2_1 = \{(u_k)_{k \geq 0} / \sum_{k=0}^\infty (1+k) |u_k|^2 < \infty \}$ and this family can be written as
$$
  Q_{\xi} = Q + v_\xi,
$$
where $v_\xi$ is given explicitely in \cite{PT}, Theorem 5.2. Using the definition of $Q$ and Lemma \ref{Link-Iso}, we see that given a potential $V \in L^2([0,1])$, there exists thus a family of potentials $V_\xi = V + \frac{v_\xi}{f^4}$ still parametrized by sequences $\xi \in l^2_1$ such that
$$
  \Delta_{g,V}(\mu) = \Delta_{g,V_\xi}(\mu), \quad \forall \mu \in \C.
$$
As a consequence of (\ref{DN-Partiel}), we obtain therefore the non-uniqueness results claimed in the case where the Dirichlet and Neumann data are measured on distinct connected components of $\partial M$. More precisely, we have
\begin{thm} \label{NonUniquenessQ3}
Let $g$ be a fixed metric as in (\ref{Metric}) and $V = V(x) \in L^2(M)$. Let $\lambda \in \R$ not belong to the Dirichlet spectra of $-\triangle_g + V$. Let $\Gamma_D, \Gamma_N$ be open subsets in distinct connected components of $\partial M$. Then for all $\xi \in l^2_1$, we have
\begin{equation} \label{t1}
  \Lambda_{g,V,\Gamma_D, \Gamma_N}(\lambda) = \Lambda_{g,V_\xi,\Gamma_D, \Gamma_N}(\lambda),
\end{equation}
where $V_\xi = \frac{v_\xi}{f^4}$ and $v_\xi$ is given in \cite{PT}, Thm 5.2. Moreover, in the case $\Gamma_D = \Gamma_0$ and $\Gamma_N = \Gamma_1$ (or the converse), the class of potentials $V_\xi$ contains all the possible potentials satisfying the property (\ref{t1}).
\end{thm}

\begin{proof}
The first part of the Theorem has been proved above. Assume now that $\Gamma_D = \Gamma_0$ and $\Gamma_N = \Gamma_1$ and that
$$
\Lambda_{g,V,\Gamma_D, \Gamma_N}(\lambda) = \Lambda_{g, \tilde{V},\Gamma_D, \Gamma_N}(\lambda). 	
$$
Thanks to  (\ref{DN-Partiel}), this is equivalent to the equalities
$$
\Delta_{g,V}(\mu_k) = \Delta_{g,\tilde{V}}(\mu_k), \quad \forall k \geq 1.
$$
Let us introduce the function $F(\mu) := \Delta_{g,V}(\mu) - \Delta_{g,\tilde{V}}(\mu)$. Then, the function $F$ can be shown to be an entire function of order $1/2$ that vanishes on the sequence $(\mu_k)_{k \geq 1}$. By the same argument as in the proof of Theorem \ref{Uniqueness}, we get $F(\mu) = 0$ for all $\mu \in \C$, \textit{i.e.}
$$
  \Delta_{g,V}(\mu) = \Delta_{g,\tilde{V}}(\mu), \quad \forall \mu \in \C.
$$
We conclude from Lemma \ref{Link-Iso} that the potential $\tilde{V}$ is isospectral to $V$. In consequence, there exists a sequence $\xi \in l^2_1$ such that $\tilde{V} = V_\xi$.
\end{proof}

\begin{rem}
1. Note that, in contrast to what is required for the counterexamples coming from the gauge invariance in Section \ref{1}, we do not assume in this result that $\overline{\Gamma_D \cup \Gamma_N} \ne \partial M$. En fact, we could have $\Gamma_D = \Gamma_0$ and $\Gamma_N = \Gamma_1$ and thus $\Gamma_D \cup \Gamma_N = \partial M$ without altering the non-uniqueness results. \\
2. We could also find counterexamples to uniqueness in the class of potentials in $L^\infty([0,1])$ but we do not have a complete characterization of such potentials as we do in $L^2([0,1])$. To handle the case of $L^\infty([0,1])$ potentials, we use the precise description of isospectral potentials in $L^2$ obtained in \cite{PT}. For instance, P\"oschel and Trubowitz showed that, given $Q = q_f + (V-\lambda) f^4 \in L^2([0,1])$and for each eigenfunction $\phi_k, \ k \geq 1$ of (\ref{Eq3}), we can find a one parameter family of explicit potentials isospectral to $Q$ in $L^2([0,1])$ by the formula
\begin{equation} \label{Iso1}
  Q_{k,t}(x) = Q(x) - 2 \frac{d^2}{dx^2} \log \theta_{k,t}(x), \quad \quad \forall t \in \R,
\end{equation} 	
where
\begin{equation} \label{Iso2}
  \theta_{k,t}(x) = 1 + (e^t - 1) \int_x^1 \phi_k^2(s) ds.
\end{equation} 		
Using the definition of $Q$, we get thus the explicit one parameter families of potentials $V_{k,t}$ isospectral to $V$:
\begin{equation} \label{IsoPot}
  V_{k,t}(x) = V(x) - \frac{2}{f^4(x)} \frac{d^2}{dx^2} \log \theta_{k,t}(x), \quad \forall k \geq 1, \quad \forall t \in \R,
\end{equation}
where $\theta_{k,t}$ is given by (\ref{Iso2}). Now it is clear from the explicit form of $V_{k,t}$ that if $V \in L^\infty([0,1])$, then $V_{k,t}$ is also in $L^\infty([0,1])$. In consequence, we have found a whole family of potentials $(V_{k,t})_{k \geq 1, t \in \R} \in L^\infty([0,1])$ such that
$$
  \Lambda_{g,V,\Gamma_D, \Gamma_N}(\lambda) = \Lambda_{g,V_{k,t},\Gamma_D, \Gamma_N}(\lambda).
$$
\end{rem}

%
%
%

\subsection{Counterexamples in the conformal class of a cylindrical warped product in dimension $n\geq 3$} \label{3}

In this Section, we construct counterexamples to uniqueness for the anisotropic Calder\'on problem \textbf{(Q2)} in dimension $n \geq 3$ modulo the gauge invariance introduced in Section \ref{1}, Definition \ref{Gauge0}. We would like first to stress the fact that it is difficult to use directly in this setting the construction of counterexamples to uniqueness for the problem \textbf{(Q3)} given in section \ref{2}. To understand why, consider two cylindrical warped product Riemannian manifolds $(M,g)$ and $(M,\tilde{g})$ with metrics $g$ and $\tilde{g}$ as in (\ref{Metric}) and consider Dirichlet and Neumann data satisfying for instance $\Gamma_D = \Gamma_0$ and $\Gamma_N = \Gamma_1$. Following the procedure given in Section \ref{2}, we would like to construct the warping functions $f$ and $\tilde{f}$ in the definition of $g$ and $\tilde{g}$ in such a way that $\Lambda_{g,\Gamma_0,\Gamma_1}(\lambda) = \Lambda_{\tilde{g},\Gamma_0,\Gamma_1}(\lambda)$. Doing so, similar arguments as in Section \ref{2} would lead to the following chain of equivalences.

\begin{lemma} \label{New1}
(1) $\Lambda_{g,\Gamma_0,\Gamma_1}(\lambda) = \Lambda_{\tilde{g},\Gamma_0,\Gamma_1}(\lambda)$ \\
iff (2) $\frac{f^{n-2}(0)}{f^n(1) \Delta_{g}(\mu_k)} = \frac{\tilde{f}^{n-2}(0)}{\tilde{f}^n(1) \Delta_{\tilde{g}}(\mu_k)}$ for all $k \geq 1$, \\
iff (3) $\Delta_{g}(\mu) = \Delta_{\tilde{g}}(\mu)$ for all $\mu \in \C$ and $\frac{f^{n-2}(0)}{f^n(1)} = \frac{\tilde{f}^{n-2}(0)}{\tilde{f}^n(1)}$, \\
iff (4) $q_f -\lambda f^4$ and $q_{\tilde{f}} - \lambda \tilde{f}^4$ are isospectral for (\ref{Eq3}) and $\frac{f^{n-2}(0)}{f^n(1)} = \frac{\tilde{f}^{n-2}(0)}{\tilde{f}^n(1)}$.
\end{lemma}

\begin{proof}
  (1) iff (2) follows from (\ref{DN-Partiel}). (2) iff (3) follows from the Complex Angular Momentum method and the universal asymptotics $ \Delta_{g}(\mu), \, \Delta_{\tilde{g}}(\mu) \sim \frac{\sinh(\sqrt{\mu})}{\sqrt{\mu}}, \quad \mu \to \infty$ (see \cite{DKN3}). (3) iff (4) follows from the proof of Lemma \ref{Link-Iso} given in \cite{DKN3}.
\end{proof}

This Lemma shows us that to construct counterexamples to uniqueness in this setting, it is enough to construct once again potentials which are isospectral to a given one $Q := q_f -\lambda f^4 \in C^\infty([0,1])$. Imagine we have found $Q_\xi = Q + v_\xi \in C^\infty([0,1])$ isospectral to $Q$ for some $\xi \in l^2_1$. It remains now to prove that there exists an $f_\xi \in C^\infty([0,1])$ such that
$$
  q_{f_\xi} - \lambda f_\xi^4 = Q_{\xi}, \quad \textrm{and} \quad f_\xi^{n-2}(0) = \left( \frac{f^{n-2}(0)}{f^n(1)} \right) f_\xi^n(1).
$$
Recalling the definition of $q_f$, this amounts to solving the nonlinear ODE with boundary conditions:
$$
  (f_\xi^{n-2})'' -Q_\xi f_\xi^{n-2} - \lambda f_\xi^{n+2} = 0, \quad \textrm{and} \quad f_\xi^{n-2}(0) = \left( \frac{f^{n-2}(0)}{f^n(1)} \right) f_\xi^n(1).
$$
Even in the case $\lambda = 0$, in which the above ODE becomes linear, the boundary conditions make it difficult to find a smooth solution on $[0,1]$. This means that in the general case, we cannot find metrics $g$ and $\tilde{g}$ of the form (\ref{Metric}) such that the condition (1) of Lemma \ref{New1} holds. The problem comes from the fact that given a metric $g$ of the form (\ref{Metric}), we are looking for counterexamples to uniqueness for \textbf{(Q2)} in the too restrictive class of metrics $\tilde{g}$ which are still of the form (\ref{Metric}). In \cite{DKN3} and below, we look instead for counterexamples to uniqueness in the full conformal class of a given metric $g$ of the form (\ref{Metric}). Precisely, we will now show that the counterexamples to uniqueness given in Theorem \ref{NonUniquenessQ3} for the anisotropic Calder\'on problem \textbf{(Q3)} lead to non trivial counterexamples to uniqueness for the anisotropic Calder\'on problem \textbf{(Q2)} in dimension $n\geq 3$ modulo the gauge invariance. To do this, we use Proposition \ref{Link-c-to-V} which gives a clear link between the anisotropic Calder\'on problems \textbf{(Q2)} and \textbf{(Q3)} when $\Gamma_D \cap \Gamma_N = \emptyset$.

Thus we work with a Riemannian manifold  $(M,g)$ of the type of a cylindrical warped product (\ref{Metric}), a smooth potential $V=V(x) \in C^\infty(M)$ and choose $\lambda \in \R$ in the complement of the Dirichlet spectrum of $-\triangle_g + V$. Given a smooth potential $\tilde{V}$ given by (\ref{IsoPot}) (it is always possible to find a large class of such \emph{smooth} potentials using Remark 3.3 in \cite{DKN3}), our goal is to show that there exist conformal factors $c$ and $\tilde{c}$ such that (see (\ref{Vgc}) for the notations)
\begin{equation} \label{fp}
  V_{g,c,\lambda} = V, \quad V_{g,\tilde{c},\lambda} = \tilde{V},
\end{equation}
and
$$
  c, \tilde{c} = 1 \ \textrm{on} \ \Gamma_D \cup \Gamma_N.
$$
If such conformal factors $c$ and $\tilde{c}$ exist, it would follow then from Theorem \ref{NonUniquenessQ3} and Proposition \ref{Link-c-to-V} that
$$
  \Lambda_{c^4 g, \Gamma_D, \Gamma_N}(\lambda) = \Lambda_{\tilde{c}^4 g, \Gamma_D, \Gamma_N}(\lambda)
$$
whenever $\Gamma_D \cap \Gamma_N = \emptyset$. Furthermore, the metrics $c^4 g$ and $\tilde{c}^4 g$ will not be gauge related according to Definition \ref{Gauge0} since they correspond to different potentials $V$  and $\tilde{V}$, as shown in Lemma \ref{lemmafactor} and the ensuing paragraph.

Next, it is easy to see using (\ref{Vgc}) that finding $c$ and $\tilde{c}$ satisfying (\ref{fp}) is equivalent to finding a smooth positive solution $w$ of the nonlinear Dirichlet problem
\begin{equation} \label{DirichletPb-q}
  \left\{ \begin{array}{rl} \triangle_g w + (\lambda - V)w - \lambda w^{\frac{n+2}{n-2}} & = 0, \ \textrm{on} \ M, \\
	                          w & = \eta, \ \textrm{on} \ \partial M,
	\end{array} \right. 														
\end{equation}
where $\eta = 1$ on $\Gamma_D \cup \Gamma_N$. Then it will follow from (\ref{Vgc}) that there will exist a $c > 0$ satisfying $V_{g,c,\lambda} = V$, $c = 1$ on $\Gamma_D \cup \Gamma_N$. As it turns out, it is again possible to find smooth positive solutions $w$ of (\ref{DirichletPb-q}) using the technique of lower and upper solutions. More precisely, the following is proved in \cite{DKN3}:

\begin{prop}[Zero frequency] \label{q-to-c-0}
  Assume that $\lambda = 0$ and $V \geq 0$ on $M$. Then for each positive smooth function $\eta$ on $\partial M$ such that $\eta = 1$ on $\Gamma_D \cup \Gamma_N$, there exists a unique smooth positive solution $w$ of (\ref{DirichletPb-q}) such that $0 <  w \leq \max \eta$ on $M$.
\end{prop}

\begin{prop}[Nonzero frequency] \label{q-to-c-1}
  1. If $\lambda >0$  and  $0 <V(x) <\lambda$ on $M$, then for each positive function $\eta$ on $\partial M$ such that $\max \eta \geq 1$ on $\partial M$, there exists a smooth positive solution $w$ of (\ref{DirichletPb-q}). \\
	2. If $\lambda < 0$ and $V(x) \geq 0$ on $M$, then for each for each positive function $\eta$ on $\partial M$ such that $\eta \leq 1$ on $\partial M$, there exists a smooth positive solution $w$ of (\ref{DirichletPb-q}). \\

	\end{prop}

We now finish the construction of counterexamples to uniqueness to \textbf{(Q2)} as follows. First, let us fix a frequency $\lambda \in \R$.

\vspace{0.2cm}
1. Assume that $\lambda >0$. Consider a potential $V = V(x) \in C^\infty(M)$ such that $0<V(x)<\lambda$ and such that $\lambda$ does not belong to the Dirichlet spectrum of $-\Delta_g +V$. This is always possible since the discrete spectrum of $-\Delta_g + V$ is unstable under small perturbations of $V$. Now, consider a potential $\tilde{V} = \tilde{V}_{k,t}(x)$ isospectral to $V$ as in (\ref{IsoPot}) and such that $0<\tilde{V}(x) < \lambda$. Observe that this can always been achieved for small enough $-\epsilon < t < \epsilon$ thanks to Remark 3.3 in \cite{DKN3}. Finally, consider a smooth positive function $\eta$ on $\partial M$ such that $\eta =1$ on $\Gamma_D \cup \Gamma_N$ and such that $\max \eta \geq 1$. Then, Proposition \ref{q-to-c-1} implies the existence of smooth positive conformal factors $c$ and $\tilde{c}$ such that
$$
  V_{g,c,\lambda} = V, \quad c = 1 \ \textrm{on} \ \Gamma_D \cup \Gamma_N,
$$
and
$$
  V_{g,\tilde{c},\lambda} = \tilde{V}, \quad \tilde{c} = 1 \ \textrm{on} \ \Gamma_D \cup \Gamma_N.
$$
But from Theorem \ref{NonUniquenessQ3}, we have
$$
  \Lambda_{g, V, \Gamma_D, \Gamma_N}(\lambda) = \Lambda_{g, \tilde{V}, \Gamma_D, \Gamma_N}(\lambda).
$$
Therefore from Proposition \ref{Link-c-to-V}, we conclude that
$$
  \Lambda_{c^4 g, \Gamma_D, \Gamma_N}(\lambda) = \Lambda_{\tilde{c}^4 g, \Gamma_D, \Gamma_N}(\lambda).
$$

\vspace{0.2cm}
2. Assume that $\lambda \leq 0$. Consider a potential $V(x)>0$ and a smooth positive function $\eta$ on $\partial M$ such that $\eta =1$ on $\Gamma_D \cup \Gamma_N$ and such that $\eta \leq 1$. Clearly, $\lambda$ does not belong to the Dirichlet spectrum of $-\Delta_g +V$. Then, we follow the same strategy as in the previous case.

\vspace{0.2cm}
It is worth repeating that the metrics $c^4 g$ and $\tilde{c}^4 g$ constructed above are not related by the new gauge invariance introduced in Section \ref{1} since they correspond - through the link (\ref{Link}) - to different potentials $V = V_{g,c,\lambda}$ and $\tilde{V} = V_{g,\tilde{c},\lambda}$. We have thus constructed a large class of counterexamples to uniqueness for the anisotropic Calder\'on problem \textbf{(Q2)} in the case where the Dirichlet and Neumann data are measured on disjoint sets of the boundary. These non-uniqueness results hold modulo this new gauge invariance.

We summarize our conclusions as follows:

\begin{thm} \label{NonUniquenessQ4}
  Let $M = [0,1] \times K$ be a cylindrical manifold having two ends equipped with a warped product metric $g$ as in (\ref{Metric}). Let $\Gamma_D, \Gamma_N$ be open sets that belong to different connected components of $\partial M$. Let $\lambda \in \R$ be a fixed frequency.  Then there exist an infinite number of smooth positive conformal factors $c$ and $\tilde{c}$ on $M$ with aren't gauge equivalent in the sense of Definition \ref{Gauge0}, and such that
$$
  \Lambda_{c^4 g, \Gamma_D, \Gamma_N}(\lambda) = \Lambda_{\tilde{c}^4 g, \Gamma_D, \Gamma_N}(\lambda).
$$	
\end{thm}

\subsection{Remarks on the case of manifolds with corners} \label{4}

In the previous non-uniqueness results for the anisotropic Calder\'on problems \textbf{(Q2)} and \textbf{(Q3)} modulo the new gauge invariance, we considered \emph{smooth} compact connected cylindrical manifolds equipped with a warped product metric and having \emph{two ends}. We indicate in this section that if we remove the assumption of smoothness for the manifold, then we can allow a connected boundary for $M$ and still obtain counterexamples to uniqueness for the Calder\'on problem.

More precisely, consider the product manifold $M = [0,1] \times K$ where $K$ is now a compact connected Riemannian manifold with boundary of dimension $n-1$. Note that the boundary of $M$ is now connected and given by:
$$
  \partial M = \left( \{0\} \times K \right) \cup \left( \{1\} \times K \right) \cup \left( (0,1) \times \partial K \right).
$$
On the other hand, we clearly lose the smoothness of the manifold since $M$ has corners. Nevertheless, we can - almost verbatim - use the previous constructions of counterexamples to uniqueness in the smooth case to construct counterexamples in this new setting.

Let us still denote the two ends of the cylinder by $\Gamma_0 = \{0\} \times K$ and $\Gamma_1 = \{1\} \times K$. The important observation is that, given Dirichlet and Neumann data $\Gamma_D, \Gamma_N$ some open subsets of $\Gamma_0 \cup \Gamma_1$ and a potential $V = V(x) \in L^\infty(M)$ or $L^2(M)$, we can construct the corresponding \emph{partial} DN map $\Lambda_{g,V,\Gamma_D,\Gamma_N}(\lambda)$ in essentially the same way as in Section \ref{3}. In fact, in this particular situation, we are still able to use separation of variables for the Dirichlet problem (\ref{Eq0-Schrodinger}) if we consider now a decomposition of the solutions onto the Hilbert basis of harmonics $Y_k$ of the Dirichlet Laplacian $-\triangle_K$ on $K$. This means that the $(Y_k)_{k \geq 1}$ are now the normalized eigenfunctions of $-\triangle_K$ associated to the \emph{Dirichlet} spectrum $(\mu_k)_{k \geq 1}$ ordered such that
$$
  0 < \mu_0 < \mu_1 \leq \mu_2 \leq \dots \leq \mu_k \to \infty.
$$
As a consequence, the DN map $\Lambda_{g,V,\Gamma_D,\Gamma_N}(\lambda)$ can be "diagonalized" in the Hilbert basis $(Y_k)_{k \geq 1}$ when $\Gamma_D, \Gamma_N \subset \Gamma_0 \cup \Gamma_1$. The last important point is to see that once it is restricted to a fixed harmonic $<Y_k>$, the DN map still has the representation (\ref{DN-Partiel}) which only involves the characteristic and Weyl-Titchmarsh functions associated to the countable family of one-dimensional boundary value problems (\ref{Eq3}) parametrized by the Dirichlet spectrum $(\mu_k)$. This does not affect the previously stated results. We thus conclude that the counterexamples to uniqueness obtained in Section \ref{2} for the problem \textbf{(Q3)} and \ref{3} for the problem \textbf{(Q2)} hold true without any change in this setting. This shows that we can remove the non-connectedness assumption for the boundary in the case of manifolds with corners.

\vspace{0.8cm}
\noindent \textbf{Acknowledgements}: The second author would like to thank Professor Shing-Tung Yau and the Center for Mathematical Sciences at Harvard University for their hospitality and financial support.  \\


\end{document}